\documentclass{article}

\usepackage{amsmath}
\usepackage[utf8]{inputenc}
\usepackage[T1]{fontenc}
\usepackage{ifthen}
\usepackage[colorlinks,bookmarks]{hyperref}
\usepackage{cleveref}



\usepackage{microtype}
\usepackage[T1]{fontenc}
\usepackage[utf8]{inputenc}

\usepackage{mathtools}
\usepackage{amssymb}

\usepackage{tikz}
\usetikzlibrary{shapes.geometric, arrows.meta, positioning, calc}

\usepackage[margin=1in]{geometry}
\usepackage{booktabs}
\usepackage{longtable}
\usepackage{xcolor}
\usepackage{natbib}
\hypersetup{colorlinks=true, linkcolor=blue!50!black,
citecolor=blue!50!black, urlcolor=blue!50!black}
\usepackage{multirow}

\usepackage{amsthm}

\newtheorem{lemma}{Lemma}
\newtheorem{proposition}{Proposition}

\theoremstyle{definition}
\newtheorem{definition}{Definition}
\newtheorem{remark}{Remark}

\newcommand{\bbR}{\mathbb{R}}

\newcommand{\refl}{R}             
\newcommand{\illum}{I}            
\newcommand{\cct}{T}              
\newcommand{\sectionspace}{E}     
\newcommand{\imagespace}{B}       
\newcommand{\proj}{\pi}           
\newcommand{\mapsec}{g}           
\newcommand{\ambig}{\Sigma}       

\begin{document}

\title{Discontinuous Prior-Mode Sections and the Geometry of Ambiguity in Intrinsic Image Decomposition}

\author{
  Ziheng Chen \footnote{stokes615@utexas.edu},
  Liangchen Liu \footnote{lcliu@utexas.edu},
  Qishi Zhan \footnote{qishi.zhan@marquette.edu},
  Minxuan Hu \footnote{mh2229@cornell.edu},
  Liang Geng \footnote{lg838@scarletmail.rutgers.edu}
}

\maketitle

\begin{abstract}
  The viral 2015 photograph known as \emph{The Dress} divides observers into
  two camps because it is ambiguous: the same image colors can be explained either as a blue--black surface under one illuminant or as a white--gold surface under another.
  We propose a geometric account in which the ambiguity arises from a \emph{singularity} in intrinsic image decomposition, the inverse problem of separating an observed image into reflectance and illumination. Our central claim is that the prior-mode section---the prior-preferred decomposition---switches across an ambiguity boundary in image space, and that any smooth learned model can only approximate this discontinuous switch by forming a thin transition layer.
  This predicts two observable signatures, where $\Delta$ is the jump between branches of the prior-mode section and $\lambda$ is the regularization strength: for inverse decomposers, an albedo Jacobian scaling as $|\Delta|/\sqrt{\lambda}$; and for forward encoders, the Fernet curvature that blows up on a scale of $1/\sqrt{\lambda}$.
  On CGIntrinsics ($N=1998$ images, $n=2\times 10^7$ pixels), the color-temperature albedo Jacobian of Careaga DPT has partial Spearman correlation $r=0.41$ with dense ground-truth albedo error, compared with $r=0.087$ and $r=0.021$ for brightness and saturation controls.
  On \emph{The Dress}, CLIP ViT-L/14 exhibits a latent curvature peak of $\kappa=73.03$ at $6473\,\mathrm{K}$, one sampled step from D65 daylight, while a control dress image peaks at $\kappa=34.75$ with no comparable feature near D65. The same characteristic appears across architectures (U-Net inverter,
  diffusion inverter, ViT encoder) and datasets (rendered indoor scenes,
  web photograph), each measured with the observable appropriate to its
  model class.
\end{abstract}


\section{Introduction}
\label{sec:intro}

The 2015 photograph known as \emph{The Dress} divides observers into two
camps of roughly equal size: one perceives the garment as blue-and-black,
the other as white-and-gold~\citep{lafer2015striking,cacoub2011dress,witzel2017most}.
This split is not measurement noise: both readings are physically
consistent, each assuming a different color temperature for the ambient
light and yielding a reflectance map that re-renders to the same pixels.
Existing Bayesian accounts~\citep{witzel2017most} locate the ambiguity in
a bimodal posterior over illuminant priors. We adopt a complementary,
sharper view: the relevant object is not the bimodality of the posterior
but the \emph{discontinuity of its mode}.

\begin{figure}[ht]
  \centering
  \includegraphics[width=\linewidth]{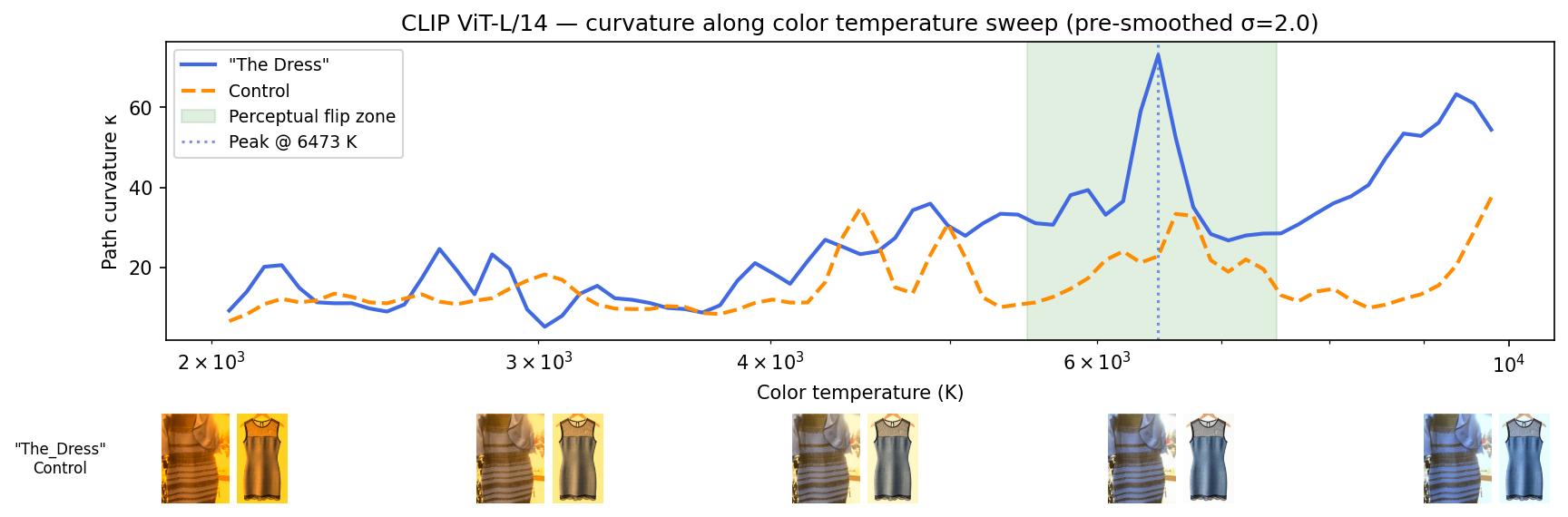}
  \caption{%
    \emph{The Dress} in CLIP: curvature peak at D65. Top: Discrete Frenet-Serret curvature of The Dress and the control image traced in CLIP latent space. The ambiguous dress exhibits a sharp peak at $\cct = 6473\,\mathrm{K}$, close to D65 ($6504\,\mathrm{K}$); the control dress shows several peaks of smaller magnitude.
    Bottom: The ambiguous and control image under perturbation across color temperatures. The second one from the right corresponds to the version under normal lighting conditions.%
  }
  \label{fig:dress_curvature}
\end{figure}

We reframe the perception of The Dress as an intrinsic decomposition task
and, simultaneously, an approximation problem in the language of
differential topology. Let $\proj\colon \sectionspace \to \imagespace$
denote the rendering map, the pixel-wise product taking
(reflectance, illumination) pairs to images; it is
non-injective, since all equal-product pairs render to the same pixel. To
fix a percept, the visual system effectively selects a single point on the
fiber $\proj^{-1}(c)$ as the mode of a prior $p$ over $(\refl,\illum)$; we
call this map the \emph{prior-mode section} $\mapsec(c)$
(\cref{def:section}).
We argue that this singular behavior reflects $\mapsec$ being single-valued and smooth on an open dense
subset of $\imagespace$ but multi-valued on a codimension-one
\emph{ambiguity locus} $\ambig \subset \imagespace$. Generic images sit off $\ambig$ and admit an unambiguous prior-mode
decomposition; The Dress sits over $\ambig$ with two perceptual camps
matching branches of $\mapsec$ on either side. The same
statement applies to any learned model that has internalized a prior,
including encoders that never train on $(\refl,\illum)$.

The singularity of the prior-mode section is inherited by any learned
inverse $f_\lambda$ that approximates $\mapsec$. An inverse trained with a
smoothness penalty of strength $\lambda$ matches $\mapsec$ off $\ambig$ but
pays a uniform price across it; this cost leaves two model-agnostic signatures. For the
one-dimensional Tikhonov model problem $L(f) = \int (f - \mapsec)^2 +\lambda\, |f'|^2$, the minimizer satisfies
\[
  \textbf{Albedo Jacobian:}\ |D f_\lambda(c_0)| \;\sim\; \frac{|\Delta \mapsec|}{\sqrt{\lambda}},
  \qquad
  \textbf{Frenet Curvature:}\ \kappa_{\gamma_\lambda}(0) \;\sim\; \frac{|\Delta \mapsec|}{\sqrt{\lambda}}
\]
across any jump $c_0$ of $\mapsec$ as $\lambda \to 0$. The $1/\sqrt{\lambda}$ exponent is inherited by deep
networks through a Gaussian-prior Laplace bridge (\cref{sec:theory}; full
derivation in Appendix~B). The Jacobian observable applies to the
decomposition setting (\cref{sec:decomp}); the curvature observable
applies to the encoder setting (\cref{sec:dress}).
\Cref{fig:global_section} provides a schematic illustration of the
sectional geometry and the resulting ill-posedness of the Jacobian.

\begin{figure}[ht]
  \centering
  \includegraphics[trim=.6cm 2.9cm .6cm 1.2cm,clip,width=0.7\textwidth]{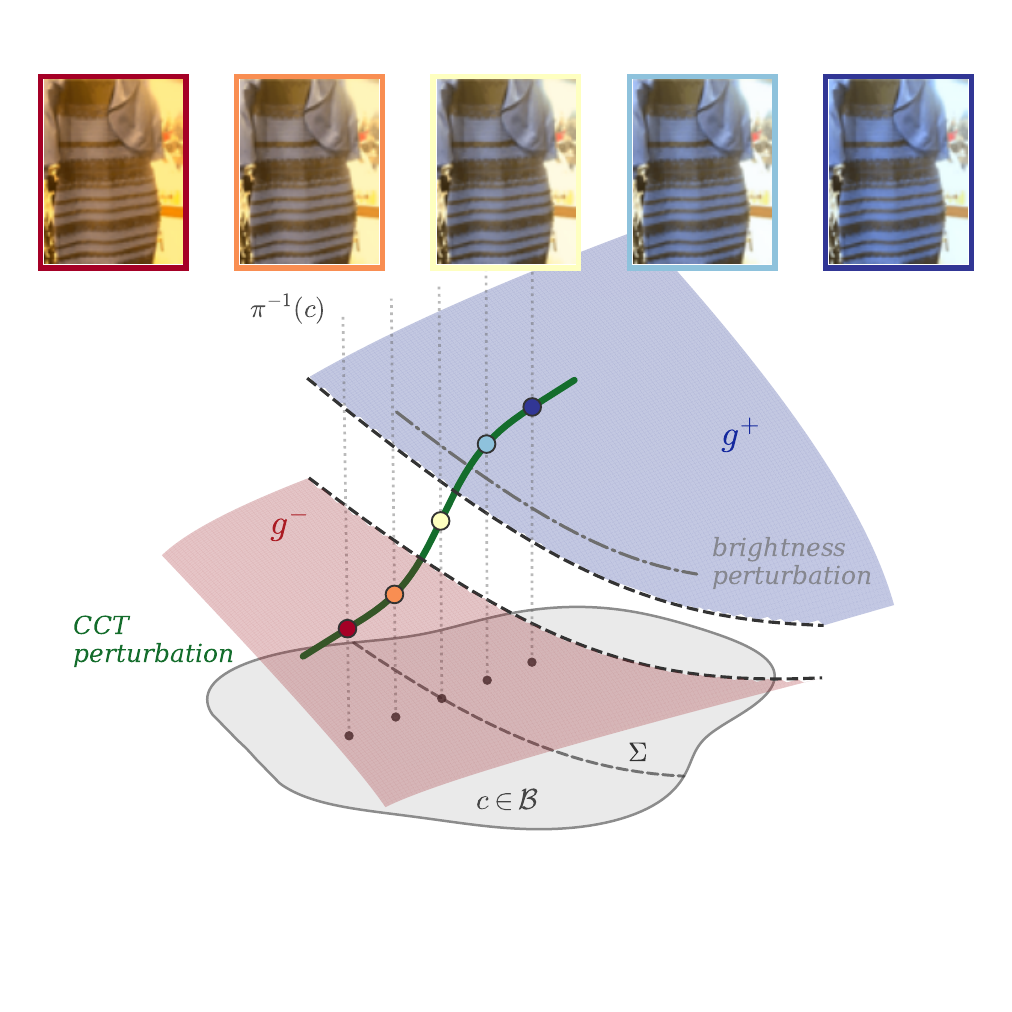}
  \caption{
    \textit{The Dress} as a prior-mode section discontinuity. The prior-mode illumination estimate $\hat I_{\mathrm{prior\ mode}}$ over the image domain $\mathcal{B}$ splits into two branches: neutral-reflectance/bright-light (blue, mode A) and high-reflectance/dim-light (pink, mode B); they meet along the ambiguity locus $\Sigma$ (pre-image shown in dashed line), a codimension-1 set where the two posterior modes have equal height. Each image $c\in\mathcal{B}$ lifts along its fiber $\pi^{-1}(c)$ (vertical dotted lines). A brightness perturbation (grey, dashed) stays on a single branch while a CCT perturbation (green, five temperaturs shown) runs normal to $\Sigma$ and crosses it, so the recovered illumination transitions sharply between branches.
  }
  \label{fig:global_section}

\end{figure}

Our contributions are three-fold:
\begin{itemize}
  \item \emph{Framework.} We identify the geometric object underlying
    intrinsic ambiguity as the discontinuity locus $\ambig$ of the
    prior-mode section $\mapsec$, rather than the bimodality of the
    posterior, and derive its $1/\sqrt{\lambda}$ and $1/\lambda$ smoothing
    signatures from a single variational principle (\cref{sec:theory}).
  \item \emph{Decomposition evidence.} On CGIntrinsics
    ($N = 1998$ valid images, pool $n = 2 \times 10^7$ pixels),
    Careaga DPT yields a partial Spearman correlation
    $r_{\cct\mid\text{ctrl}} = 0.409$ between the per-pixel albedo
    Jacobian along color temperature and dense ground-truth error,
    compared with $r_{\text{bright}} = 0.087$ and $r_{\text{sat}} = 0.021$
    at matched perturbation magnitudes (Wilcoxon
    $r_{\cct}\!>\!r_{\text{bright}}$, $z=11.0$). The
    qualitative ranking replicates on Intrinsic Anything ($N = 100$,
    partial $r = 0.287 / 0.098 / 0.033$).
  \item \emph{Dress evidence.} For The Dress, CLIP ViT-L/14 traced along
    a color-temperature sweep exhibits a discrete Frenet--Serret
    curvature $\kappa = 73.03$ at $\cct \approx 6473\,\mathrm{K}$, with
    a near-isotropic bend. Control dresses peak at $\kappa = 34.75$
    with no comparable feature near D65. This is the curvature observable
    detecting $\ambig$ in a setting where the Jacobian observable cannot
    (\cref{sec:dress}).
\end{itemize}

\paragraph{Related work.}
Our framing sits at the intersection of three lines of work. First, on
the Dress and color constancy, prior accounts explain the bistable
percept through bimodal Bayesian inference over illuminant priors
\citep{witzel2017most,lafer2015striking} or through individual differences
in chromatic adaptation (\cite{macadam1956chromatic}). We argue that the
discontinuity of prior-mode section
makes the phenomenon detectable in any model carrying a prior, even without explicit decomposition training.
Second, on intrinsic image decomposition, classical Retinex
(\cite{land1971lightness}), SIRFS \citep{shelhamer2015scene}, and modern learned
decomposers (such as CGIntrinsics (\cite{li2018cgintrinsics}) and the
Careaga DPT pipeline \citep{careaga2024colorful}) implicitly select a single
branch of $\mapsec$; their failure modes near $\ambig$ have been reported
as ``ambiguous shading'' but, to our knowledge, not localized to a
codimension-one locus with a scaling law. Third, the $1/\sqrt{\lambda}$ jump-smoothing rate is the classical
Tikhonov boundary-layer estimate (\cite{engl1996regularization}), and the
codimension-one mode-switching surface is the Maxwell set of singularity
theory \citep{arnold1986catastrophe}. Connecting these to learned inverses
via a Gaussian-prior Laplace bridge predicts the same exponents in
deep nets, and read them off two unrelated systems (a U-Net decomposer and
a frozen CLIP encoder \cite{radford2021learning}) at the predicted
location.

\paragraph{Outline.}
\Cref{sec:theory} formalizes the prior-mode section, the ambiguity locus
$\ambig$, and the two scaling signatures (proved in Appendices~A and B).
\Cref{sec:decomp,sec:dress} report the decomposition and CLIP-curvature
experiments, and \cref{sec:discussion} closes with the shape of $\ambig$,
limitations, and implications for mode-aware training.

\section{Prior-Mode Sections and the Price of Smoothing}
\label{sec:theory}

\subsection{Problem Setup}
The prior mode is defined following the terminology of \cite{husemoller1966fibre}.

\begin{definition}
  \label{def:section}
  Let $\sectionspace=\mathbb{R}_+^{d}$ denote the \emph{physical space} of pairs $(\refl,\illum)$
  (reflectance, illumination) and $\imagespace$ the \emph{image space} of observed
  color signals $c$.  The \emph{rendering map}
  $\proj\colon\sectionspace\to\imagespace$, defined as a per-pixel product map $c = \proj(\refl_1,\dots, \refl_{d-1},\illum) := (\refl_1 \illum, \dots, \refl_{d-1}\illum)$, is a smooth
  submersion.
  For any given prior $p$ on $\sectionspace$ (encoding illuminant chromaticity and reflectance
  statistics), it induces the \emph{prior-mode section}
  \begin{equation}\label{eq:mapsec}
    \mapsec(c) \;:=\;
    \operatorname*{arg\,max}_{\illum\ \text{s.t.}\ \proj(\refl,\illum)=c} p(\refl,\illum).
  \end{equation}
\end{definition}

\Cref{fig:global_section} illustrates the geometry. Let
$\ambig \subset \imagespace$ be the locus where two branches $\mapsec^{\pm}$
tie for the maximum, so that $\mapsec$ jumps across it; this is the
\emph{Maxwell set} of $p$~\citep{arnold1986catastrophe} and is generically
codimension one. We call $\ambig$ the \emph{ambiguity locus}, over which The
Dress sits, its two perceptual camps being $\mapsec^+$ and $\mapsec^-$.

Although there is no topological obstruction to a global section (since $\sectionspace$ is a smooth fiber bundle over $\imagespace$, per \cite{hatcher2002algebraic}), the prior-mode section in \cref{eq:mapsec} need not be globally single-valued; \Cref{sec:gaussian-mixture} establishes the existence of $\ambig$ under log-Gaussian mixture priors.

\subsection{A Tikhonov model problem}
\label{subsec:model}

We model the effect of learning a smooth approximation to a nonsmooth
prior-selected explanation by a Tikhonov functional on the data manifold.
Let $\mathcal M := \proj(\operatorname{supp}p)\subseteq \imagespace$
denote the manifold of physically realizable images, and let
$\mapsec:\mathcal M\to\bbR^q$ denote the ideal prior-selected output.
For $f:\mathcal M\to\bbR^q$, define
\begin{equation}\label{eq:tikhonov}
  L(f) \;:=\;
  \int_{\mathcal{M}}
  |f(c)-\mapsec(c)|^2
  +\; \lambda\,\bigl\|\nabla_{\!\mathcal{M}} f(c)\bigr\|_F^2
  \;\mathrm{d}\mathrm{vol}(c),
\end{equation}
where $\lambda>0$ is the smoothness strength. The unique minimizer
$f_\lambda \in H^1(\mathcal{M};\bbR^d)$ satisfies the Euler--Lagrange
equation
$
f_\lambda - \lambda \Delta_{\mathcal{M}} f_\lambda = \mapsec.
$
\Cref{app:manifold_el,app:fermi_reduction,app:tikhonov_minimiser} detail the derivation of the approximate solution
$f_\lambda = \mapsec \ast K_\lambda + \mathcal{O}(\sqrt{\lambda})$ with the Laplace kernel
$
K_\lambda(n)
= \frac{1}{2\sqrt{\lambda}}\,e^{-|n|/\sqrt{\lambda}}
$
of effective width $\sqrt{\lambda}$. Thus the regularized model smooths any nonsmooth transition across $\ambig$ over a normal layer of width $O(\sqrt{\lambda})$.

For inverse decomposers, the relevant singularity is a jump in the
prior-selected reflectance or illumination explanation.  The smoothing
layer then forces a large normal Jacobian.

\begin{proposition}[Value jump implies Jacobian concentration]
  \label{prop:jac}
  Let $c_0\in\ambig$ be a regular point, and suppose that $\mapsec$ is
  $C^1$ on the two sides of $\ambig$ with one-sided limits
  $\mapsec^\pm(c_0)$; let $\Delta\mapsec(c_0):=\mapsec^+(c_0)-\mapsec^-(c_0)$ denote the jump.
  Then, in the normal coordinate $n$ centered at $c_0$,
  \[
    |\partial_n f_\lambda(c_0)|
    =
    \frac{|\Delta\mapsec(c_0)|}{2\sqrt{\lambda}}
    +O(1) \to\infty,
    \qquad
    \lambda\to 0.
  \]
  Away from $\ambig$, $f_\lambda$ converges locally to $\mapsec$
  with its first derivatives on each smooth branch.
  Furthermore, we have the Jacobian-error link relationship $| f_\lambda(n) - \mapsec(n) | = \sqrt{\lambda} \left(|f_\lambda'(n)| + \mathcal{O}(1)\right)$.
\end{proposition}

For forward encoders, the output is not a reflectance estimate, so the
value-jump observable above is not directly available.  The latent
representation may instead remain continuous as the image crosses
$\ambig$, while its direction of variation changes because the encoder
moves between two competing explanatory branches.  The corresponding local
singularity is therefore a \emph{latent kink}: a continuous path with a
jump in tangent direction.  The same Tikhonov smoothing mechanism then
predicts curvature blow-up.

\begin{proposition}[Latent kink implies curvature concentration]
  \label{prop:kink_curv}
  Let $\gamma_0:\bbR\to\bbR^Q$ be a continuous path that is $C^2$ on the
  two sides of $0$:
  \[
    \gamma_0(n)=
    \begin{cases}
      \gamma_-(n), & n<0,\\
      \gamma_+(n), & n>0,
    \end{cases} \qquad  \gamma_-(0)=\gamma_+(0).
  \]
  Let
  $
  v_-:=\gamma_-'(0), \  v_+:=\gamma_+'(0),
  \  v_m:=\frac{v_-+v_+}{2}, \ \Delta v:=v_+-v_-
  $
  with $\gamma_\lambda=K_\lambda*\gamma_0$ being the one-dimensional Tikhonov
  smoothing.
  If $v_m\neq 0$ and
  the tangent jump has a nonzero transverse component, $ P_{v_m^\perp}\Delta v\neq 0$ where $P_{v_m^\perp}$ denotes orthogonal projection onto the complement of
  $v_m$.
  then the Frenet curvature of $\gamma_\lambda$ satisfies
  \[
    \kappa_{\gamma_\lambda}(0)
    =\frac{\|P_{v_m^\perp}\Delta v\|}
    {2\sqrt{\lambda}\,\|v_m\|^2} +O(1)\to\infty,
    \qquad \lambda\to 0,
  \]
  Hence smoothing a latent tangent switch produces a localized
  curvature peak of order $\lambda^{-1/2}$.
\end{proposition}

The two propositions describe different observables for the same geometric
event.  Inverse decomposers expose the ambiguity boundary through a large
albedo Jacobian, while forward encoders expose it through curvature of a
latent path when the competing latent branches meet with different tangent
directions (numerically verified in \cref{sec:dress}).
Detailed proofs are provided in Appendix~\ref{app:sobolev}.



The regularization term in the Tikhonov problem plays the role of a weight penalty in neural-network training. For shallow feed-forward neural networks, \cite{Rahul2021JMLR} constructs a family of semi-norms in the second-order Radon-domain bounded-variation space, closely related to the Sobolev norm $H^1$ we impose in the Tikhonov model problem.
For deep networks, \citet{rahaman2019spectral} shows an analogous implicit spectral
bias that suppresses high-frequency components and plays the role of~$\lambda$ in~\cref{eq:tikhonov}.
We defer the connection between the neural-network training loss and the Tikhonov problem to Appendix~\ref{app:loss_taxonomy}.

\section[Ambiguity locus in Intrinsic Decomposers: the Albedo Jacobian]{$\ambig$ in Intrinsic Decomposers: the Albedo Jacobian}
\label{sec:decomp}

This section provides the main empirical support for \cref{sec:theory}. We
probe $\ambig$ in the intrinsic decomposition setting via the per-pixel
albedo Jacobian under illuminant perturbation, and confirm that the
pattern survives an architecture change.

\paragraph{Dataset and models.}
We use CGIntrinsics~\citep{li2018cgintrinsics}, a large-scale rendered
dataset that supplies ground-truth albedo (reflectance $\refl$) and
shading $\illum$ for indoor scenes under physically-based illumination.
Our anchor model is the feed-forward DPT-style U-Net of
\citet{careaga2023intrinsic}, trained with an MSE-family loss on
ground-truth albedo: this is the closest empirical analogue to the
Tikhonov model problem of \cref{sec:theory}, with MSE matching the
quadratic fidelity and weight decay acting as an implicit Sobolev
penalty. For an architecture contrast we include Intrinsic Anything
(IA) of \citet{chen2024intrinsicanything}, a latent-diffusion model.

\paragraph{Setup.}
For each image $c$ in CGIntrinsics we construct three families of
perturbed inputs and approximate the Jacobian by central differences:
\begin{itemize}
  \item \textbf{CCT}: white-balance shift~\citep{kang2006computational}
    at four reference color temperatures
    $T \in \{3500, 5000, 6500, 8000\}\,\mathrm{K}$, each probed at
    $T \pm 500\,\mathrm{K}$ ($\delta = 500\,\mathrm{K}$); this is the
    illuminant-prior axis aligned with $\ambig$.
  \item \textbf{Brightness}: uniform linear-RGB scaling $c \mapsto \alpha c$
    at $\alpha \in \{0.70, 0.85, 1.00, 1.15\}$, $\delta = 0.10$.
  \item \textbf{Saturation}: HSV saturation scaling at the same four
    reference scales, $\delta = 0.10$.
\end{itemize}
We run Careaga on each perturbed image, compute the per-pixel
central-difference Jacobian norm $\|\partial \refl / \partial(\text{perturb})\|_{2}$
at each reference point, and take the pixel-wise maximum across the four
reference points as a single sensitivity map per axis. Albedo error
$e := \|\refl_{\text{pred}} - \refl_{\text{GT}}\|_2$ is computed
per pixel on the unperturbed image. The per-image scalars
$r_{\cct}$, $r_{\text{bright}}$, $r_{\text{sat}}$ are Spearman
correlations between each sensitivity map and $e$, restricted to
foreground pixels. Of $2000$ sampled images, $N = 1998$ pass the
foreground-mask threshold; the pooled analysis uses $n = 2 \times 10^7$
pixels.

\paragraph{Marginal vs.\ partial correlation.}
The three perturbation axes are mutually correlated since a white-balance shift
changes chromaticity but also perturbs apparent brightness and
saturation. A raw (\emph{marginal}) Spearman correlation between CCT-sensitivity and error
is therefore inflated by whatever CCT shares with the brightness and
saturation channels. The \emph{partial} Spearman correlation
$r_{\cct\,\cdot\,\text{bright},\text{sat}}$ removes this confound by regressing out
the other two sensitivity channels from both and thus isolates the residual association.

\paragraph{Results.}
\Cref{tab:partialr} reports marginal and partial Spearman correlations
between per-image sensitivity and per-image error.
CCT dominates in both the marginal and partial analyses. Once CCT is
controlled, brightness and saturation carry essentially no residual
signal ($r = 0.087$ and $0.021$), identifying illuminant chromaticity as
the operative axis. Paired Wilcoxon signed-rank tests give
$z = 11.0$ for $r_{\cct} > r_{\text{bright}}$
(mean difference $\Delta\bar{r} = 0.060$) and
$z = 13.2$ for $r_{\cct} > r_{\text{sat}}$ ($\Delta\bar{r} = 0.098$).
Per \Cref{prop:jac}, error and $|Df_\lambda|$ both blow up
near $\ambig$, parameterized by illuminant chromaticity.
The CCT-dominant partial-$r$ is much stronger compared to the brightness and saturation controls, indicating that images closer to $\ambig$ along the
illuminant axis show both higher Jacobian magnitude and higher
reconstruction error.

\begin{table}[ht]
  \centering
  \caption{Spearman correlation between per-image albedo sensitivity and
    albedo error under three illuminant-axis perturbations, for two
    backbones on CGIntrinsics.
    Significance:
    $^{***}\,p<0.001$, $^{**}\,p<0.01$, $^{*}\,p<0.05$, taking $N$ images
  as the effective sample size to account for within-image pixel dependence.}
  \label{tab:partialr}
  \begin{tabular}{llccc}
    \toprule
    Model & Metric & CCT & Brightness & Saturation \\
    \midrule
    \multirow{2}{*}{Careaga DPT ($N=1998$)} & Marginal $r$ & $\textbf{0.547}$$^{***}$ & $0.399$$^{***}$ & $0.313$$^{***}$ \\
    & Partial $r$ & $\mathbf{0.409}$$^{***}$ & $0.087$$^{***}$ & $0.021$ \\
    \midrule
    \multirow{2}{*}{IntrinsicAnything ($N=100$)} & Marginal $r$ & $\textbf{0.410}$$^{***}$ & $0.318$$^{**}$ & $0.212$$^{*}$ \\
    & Partial $r$ & $\mathbf{0.287}$$^{**}$ & $0.098$ & $0.033$ \\
    \bottomrule
  \end{tabular}
\end{table}

\Cref{fig:binplot} bins the $N = 1998$ images by sensitivity quantile
and plots mean error per bin for each axis. The CCT curve rises
steadily while the brightness and saturation curves are markedly flatter in the later bins. This indicates that the CCT sensitivity is a better proxy for model error.
We also put the per-image case study (\cref{fig:heatmaps2}) in Appendix~\ref{careaga}.

\begin{figure}[ht]
  \centering
  \includegraphics[width=0.72\linewidth]{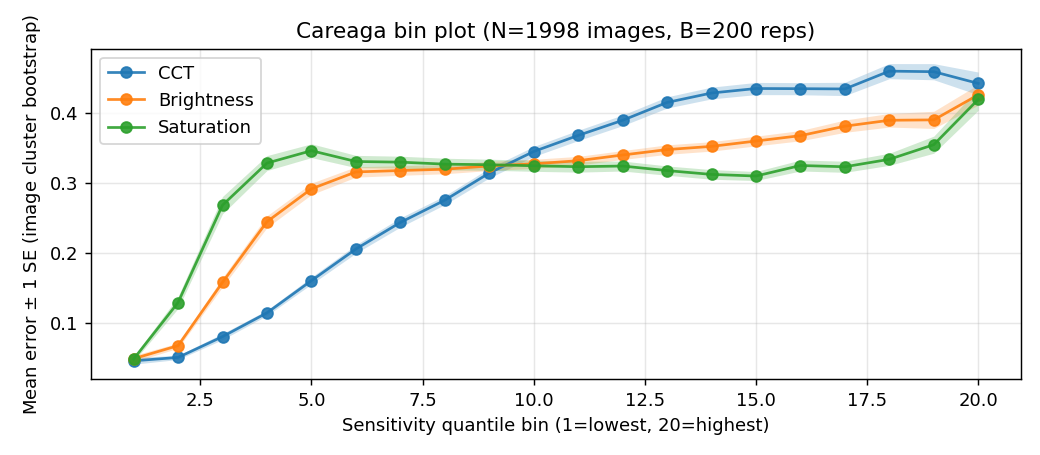}
  \caption{Mean error vs.\ sensitivity-quantile bin for CCT (blue),
    brightness (orange), and saturation (green) perturbations. Each bin
  pools $1/20$ of the 20M pixels. The CCT mean error rises steadily while the other two controls are much flatter.}
  \label{fig:binplot}
\end{figure}

\paragraph{Replication in a diffusion inverter.}
To check that the CCT-dominant ordering is not an artefact of a single
feed-forward backbone, we re-run the protocol on Intrinsic Anything (IA;
\cite{chen2024intrinsicanything}), a latent-diffusion model that differs
from Careaga in both architecture and training, on $N = 100$ held-out
CGIntrinsics images (compute-limited, hence higher variance). The partial-$r$
ordering is preserved (CCT-dominant; \cref{tab:partialr}): the
$r_{\cct} > r_{\text{sat}}$ gap is significant (Wilcoxon $z = 3.50$), while
$r_{\cct} > r_{\text{bright}}$ does not reach significance at this $N$
($p = 0.46$), consistent with the smaller effect size and reduced power.

\section[Ambiguity locus in CLIP Latent Space: Curvature at The Dress]{$\ambig$ in CLIP Latent Space: Curvature at The Dress}
\label{sec:dress}

\Cref{sec:theory,sec:decomp} established the Jacobian observable on an
inverse model (image $\to$ reflectance). We now turn to the \emph{curvature}
observable on a forward encoder and show that the same ambiguity locus
$\ambig$ is detectable there.

\paragraph{Setup.}
\label{sec:dress:setup}

We use CLIP ViT-L/14 (OpenAI weights, \cite{radford2021learning}) as the
encoder. The image space $\imagespace$ is parameterized by color
temperature $\cct \in [3000, 10000]\,\mathrm{K}$: for each $\cct$ we apply a
Bradford chromatic-adaptation transform (CAT, \cite{hunt1997reversing}) from
D65 ($6504\,\mathrm{K}$) to $\cct$, embed the adapted image with CLIP, and
$\ell^2$-normalize the resulting $768$-dimensional vector. This yields a
smooth path $e(\cct) \in \mathbb{S}^{767}$. The sweep uses $75$ log-spaced
steps; curvature quantities are evaluated on Gaussian-smoothed embeddings
($\sigma = 2$ steps) to suppress numerical noise. As a control, we apply
the identical CAT sweep to a Gemini-generated dress with a neutral blue
tone.
To quantify how sharply the latent trajectory bends, we use the discrete
Frenet-Serret curvature
$
\kappa_i := \|\Delta\hat{t}_i\| / \tfrac{1}{2}(\|\Delta e_{i-1}\|+\|\Delta e_i\|)
$
where $\hat{t}_i = \Delta e_i / \|\Delta e_i\|$ is the unit tangent and
$\Delta\hat{t}_i = \hat{t}_{i+1} - \hat{t}_i$. This is the reciprocal of the
local osculating-circle radius and, crucially, is invariant to monotone
reparameterizations of the path, thus isolating direction change from the
encoder's speed warp along the CCT axis.

\paragraph{Curvature peak at D65.}
\label{sec:dress:curvature}

\Cref{fig:dress_curvature} shows the main result. The Dress exhibits a
single prominent curvature peak at $\cct = 6473\,\mathrm{K}$
($\kappa = 73.03$), within one sweep step of the canonical daylight
illuminant D65 ($6504\,\mathrm{K}$) under which the original photograph was taken.
By contrast, the control dress shows only modest, off-D65 peaks. The
same pattern recurs across three further encoders; see
\cref{fig:dress_curvature_full} in the appendix.

\section{Discussion and Conclusion}
\label{sec:discussion}

\paragraph{$\ambig$ and the geometric picture.}
The locus $\ambig$ is determined by the training prior $p(\refl,\illum)$,
not architecture: for both decomposers it lies along the warm-to-daylight
illuminant boundary (\cref{sec:decomp}); for CLIP the curvature peak at
D65 (\cref{sec:dress}) reflects a bimodal illuminant prior.  That the
same locus surfaces across architectures and observables is the central
empirical finding.  Recasting The Dress as a singularity of
$\mapsec$ makes it observable in any learned
system, and the $1/\!\sqrt{\lambda}$, $1/\lambda$ scalings of
\cref{sec:theory} promote it from qualitative failure mode to
quantitative signature.

\paragraph{Conclusion and implication for mode-aware training.}
The bistable percept is the human-visible tip of a codimension-one
discontinuity that any approximation $f_\lambda$ inherits as a Jacobian
or curvature spike, verified here on two unrelated decomposers and a
frozen CLIP encoder at the predicted locus.  Because the albedo Jacobian
proxies prediction error, downweighting high-Jacobian samples under CCT
perturbation provides a principled, mode-aware alternative to uniform
regularization.

\ifx\tagdsmode\tagdssubmission
\else
\acks{Acknowledgements go here.}
\fi

\bibliographystyle{plainnat}
\bibliography{ref}

\appendix
\section{From Training Losses to Prior-Mode Section}
\label{app:loss_taxonomy}

This appendix supports \cref{subsec:model} by showing that
standard training objectives recover the prior-mode section
$\mapsec$, and that empirical risk minimisation (ERM) on finite
data preserves this conclusion.

\paragraph{Population minimiser}
\label{app:pop_min}

When the label $y_i = \mapsec(c_i)$ is deterministic and noise-free, the population minimiser of any proper scoring rule (including MSE and cross-entropy) is $f^* = \mapsec$ a.e.\,.
Standard ERM consistency (e.g.\ via Rademacher complexity bounds; \citealt{bartlett2002rademacher}) ensures that with sufficient data the empirical minimiser targets~$\mapsec$.
The crucial corollary is that $f^*$ is the prior mode rather than the posterior mean. The posterior is bimodal near~$\ambig$ and the posterior mean lies between the two modes (typically off the fibre) among which $f^*=\mapsec$ selects one of them.

\paragraph{CLIP (InfoNCE)}
\label{app:clip}

CLIP's InfoNCE loss lower-bounds mutual information under a
Gaussian mixture model for the embeddings.  Under that prior,
$\log p(\refl,\illum\mid c)$ is quadratic and
$\mapsec(c)=\operatorname*{arg\,max}p$ is the nearest mixture
centre.  InfoNCE pushes each embedding toward its conditional MAP
centre; the limiting encoder therefore approximates the prior-mode
section for the Gaussian-mixture prior.

\paragraph{Careaga (Ordinal shading MSE)}
\label{app:careaga}

\cite{careaga2023intrinsic} train with a direct MSE on the ordinal ground-truth shading and the albedo is an immediate consequence by dividing the observed image pixels with the recovered shading map.
Thus, the model problem can be readily applied to the shading prediction and thus the singularity at the ambiguous locus translates to the reflectance prediction.
\section{Tikhonov Model: Full Derivations}
\label{app:sobolev}

This appendix proves \cref{prop:jac,prop:kink_curv} from
\cref{sec:theory}.  The argument proceeds in three
stages: we first derive the Euler-Lagrange equation on the data manifold
$\mathcal{M}$ (\cref{app:manifold_el}), then show that it reduces to a
one-dimensional ODE in the normal direction to $\ambig$
(\cref{app:fermi_reduction}), and finally solve the one-dimensional problem
explicitly via Fourier analysis (\cref{app:tikhonov_minimiser}).  The
proposition proofs in
\cref{app:jac_proof,app:curve} use only the one-dimensional
convolution formula.

\subsection[Euler-Lagrange Equation on M]{Euler-Lagrange Equation on $\mathcal{M}$}
\label{app:manifold_el}

Let $(\mathcal{M},G)$ be an orientable, smooth, and compact Riemannian manifold of dimension $m$.  For
$f\colon\mathcal{M}\to\bbR$ in $H^1(\mathcal{M};\bbR)$ and any given $\mapsec\in L^2(\mathcal{M};\bbR)$, the Tikhonov
functional \eqref{eq:tikhonov} reads
\[
  L(f) \;=\; \int_{\mathcal{M}}
  |f(c)-\mapsec(c)\bigr|^2
  +\, \lambda\,\bigl\|\nabla_{\!\mathcal{M}} f(c)\bigr\|_F^2
  \;\mathrm{d}\mathrm{vol}(c),
\]
where the Frobenius norm of the Jacobian is
$\|\nabla_{\!\mathcal{M}} f\|_F^2 = G^{ij}\,\partial_i f\,\partial_j f$
in local coordinates.  The functional is
strictly convex and coercive on $H^1$ for $\lambda>0$; existence and uniqueness of a
minimiser follow by direct methods in the calculus of
variations~\cite{evans2022partial}.

\begin{lemma}[Euler-Lagrange equation]\label{lem:el}
  The minimiser $f_\lambda$ satisfies
  \begin{equation}\label{eq:el_manifold}
    f_\lambda \;-\; \lambda\,\Delta_{\mathcal{M}}\, f_\lambda
    \;=\; \mapsec,
  \end{equation}
  in the weak sense on $\mathcal{M}$, where
  $\Delta_{\mathcal{M}} = \frac{1}{\sqrt{\det G}}\,\partial_i\!\bigl(
  \sqrt{\det G}\;G^{ij}\,\partial_j\bigr)$
  is the Laplace-Beltrami operator.
  If $\partial\mathcal{M}\neq\emptyset$,
  the natural boundary condition is
  Neumann: $\partial_\nu f_\lambda = 0$ on $\partial\mathcal{M}$,
  where $\nu$ is the normal vector pointing outwards.
\end{lemma}

\begin{proof}
  Take a perturbation $\eta \in H^1(\mathcal{M};\bbR)$ and the stationary condition follows the vanishing first variation:
  \begin{align*}
    0 \;=\; \frac{d}{d\varepsilon}\Big|_{\varepsilon=0}
    L(f_\lambda + \varepsilon\eta)
    &\;=\; 2\int_{\mathcal{M}}
    \bigl\langle f_\lambda - \mapsec,\,\eta\bigr\rangle
    +\, \lambda\,
    G^{ij}\,\partial_i f_\lambda\,\partial_j \eta
    \;\mathrm{d}\mathrm{vol}.
  \end{align*}
  Integrating the gradient term by parts on $\mathcal{M}$ (Riemannian
  divergence theorem), the boundary integral vanishes for $\eta$
  unconstrained on $\partial\mathcal{M}$ iff
  $\partial_\nu f_\lambda = 0$, and the bulk integral gives
  \[
    \int_{\mathcal{M}}
    \bigl\langle
    f_\lambda - \mapsec - \lambda\,\Delta_{\mathcal{M}} f_\lambda,
    \;\eta
    \bigr\rangle \;\mathrm{d}\mathrm{vol}
    \;=\; 0
    \qquad \forall\,\eta,
  \]
  which yields \eqref{eq:el_manifold} a.e.
\end{proof}


\subsection{Boundary-Layer Reduction via Fermi Coordinates}
\label{app:fermi_reduction}

Let $\ambig \subset \mathcal{M}$ be a smooth hypersurface
(codimension one, which will be motivated by a concrete example in \cref{sec:gaussian-mixture}).  By the tubular neighbourhood theorem, there exists
$\delta_0 > 0$ such that the map $(n,\tau) \mapsto \exp_\tau(n\,\nu_\tau)$
is a diffeomorphism from $(-\delta_0,\delta_0)\times\ambig$ onto a tubular
neighbourhood $U_{\delta_0}$ of $\ambig$ in $\mathcal{M}$.  Here $n$ is
signed distance to $\ambig$,
$\tau \in \ambig$ is the footpoint, and
$\nu_\tau$ the unit normal at $\tau$.  In these Fermi normal
coordinates, the metric takes the form
\begin{equation}\label{eq:fermi_metric}
  ds^2 \;=\; dn^2 \;+\; h_{ij}(n,\tau)\,d\tau^i\,d\tau^j,
\end{equation}
where $h_{ij}(0,\tau) = \bar{G}_{ij}(\tau)$ is the induced metric on
$\ambig$ (\cite{gray2003tubes}).

\begin{lemma}[Laplace-Beltrami decomposition]\label{lem:lb_decomp}
  In Fermi coordinates \eqref{eq:fermi_metric}, the Laplace-Beltrami
  operator on $\mathcal{M}$ decomposes as
  \begin{equation}\label{eq:lb_fermi}
    \Delta_{\mathcal{M}} \;=\;
    \partial_n^2
    \;+\; H(n,\tau)\,\partial_n
    \;+\; \Delta_{\ambig_n},
  \end{equation}
  where
  $H(n,\tau) = \partial_n \log\sqrt{\det h(n,\tau)}$
  is the mean curvature of the level set $\ambig_n := \{n = \mathrm{const}\}$,
  and $\Delta_{\ambig_n}$ is the intrinsic Laplacian on $\ambig_n$.
\end{lemma}

\begin{proof}
  In coordinates $(x^0,x^1,\dots,x^{m-1}) = (n,\tau^1,\dots,\tau^{m-1})$,
  the full metric is block-diagonal:
  $G_{00} = 1$, $G_{0i} = 0$, $G_{ij} = h_{ij}$ for $i,j\ge 1$.
  Hence $\det G = \det h$, $G^{00} = 1$, and
  \begin{align*}
    \Delta_{\mathcal{M}} f
    &= \frac{1}{\sqrt{\det h}}\;\partial_n\!\bigl(\sqrt{\det h}\;\partial_n f\bigr)
    \;+\; \frac{1}{\sqrt{\det h}}\;\partial_i\!\bigl(
    \sqrt{\det h}\;h^{ij}\,\partial_j f\bigr) \\
    &= \partial_n^2 f
    \;+\; \frac{\partial_n\sqrt{\det h}}{\sqrt{\det h}}\;\partial_n f
    \;+\; \Delta_{\ambig_n} f.
  \end{align*}
\end{proof}

\begin{proposition}[Normal-direction dominance]\label{prop:bl}
  Let $\mapsec$ be $L^2$ on $\mathcal{M}$ and $C^1$ on $\mathcal{M}\setminus\ambig$ with a jump
  discontinuity $\Delta\mapsec$ across $\ambig$.  Then, as
  $\lambda\to 0$, the minimiser $f_\lambda$ of \cref{eq:tikhonov} satisfies,
  within the transition layer $|n| \lesssim \sqrt{\lambda}$:
  \begin{align}
    \label{eq:scaling}
    \partial_n^2 f_\lambda \sim O(\lambda^{-1}), \quad    H\,\partial_n f_\lambda \sim O(\lambda^{-1/2}),  \quad  \Delta_{\ambig_n}\, f_\lambda \sim O(1).
  \end{align}
  Consequently, the Euler-Lagrange equation \cref{eq:el_manifold}
  reduces to
  \begin{equation}\label{eq:el_1d}
    f_\lambda \;-\; \lambda\,\partial_n^2 f_\lambda
    \;=\; \mapsec + O\!\bigl(\sqrt{\lambda}\bigr)
  \end{equation}
  in the transition layer.  Outside the layer ($|n| \gg \sqrt\lambda$),
  $f_\lambda \to \mapsec$ pointwise.
\end{proposition}

\begin{proof}
  We start by a scaling argument.
  The target $\mapsec$ jumps by $O(|\Delta\mapsec|)$ across $\ambig$
  but is smooth along $\ambig$.  The minimiser $f_\lambda$ interpolates
  this jump over a layer of width $\delta = O(\sqrt{\lambda})$
  (as determined a~posteriori by the one-dimensional solution in
  \cref{app:tikhonov_minimiser}).  Within this layer, denote the
  characteristic size of $f_\lambda - \mapsec$ by $A = O(|\Delta\mapsec|)$.
  Meanwhile, the mean-curvature term $H(n,\tau) = H(0,\tau) + O(n)$ is bounded on the layer, so
  $H\,\partial_n f_\lambda \sim
  O(1) \cdot O(|\Delta\mapsec|/\sqrt\lambda)
  = O(\lambda^{-1/2})$.

  For the normal derivatives,
  the function changes by $O(A)$ over distance $O(\delta)$ in the
  $n$-direction, so
  $\partial_n f_\lambda \sim A/\delta \sim |\Delta\mapsec|/\sqrt\lambda$
  and
  $\partial_n^2 f_\lambda \sim A/\delta^2 \sim |\Delta\mapsec|/\lambda$.
  While for the tangential derivatives
  along $\ambig$, the jump amplitude $|\Delta\mapsec|$ and the smooth
  branches $\mapsec^\pm$ vary on a scale $\ell$ independent of $\lambda$
  (set by the geometry of $\ambig$ and the prior $p$).  The tangential
  derivatives of $f_\lambda$ inherit this smooth variation by
  $\Delta_{\ambig_n} f_\lambda \sim A/\ell^2 \sim O(1)$.

  Finally, substituting into \eqref{eq:el_manifold} via \eqref{eq:lb_fermi},
  the $\lambda\,\partial_n^2$ term is $O(1)$ (same order as $f_\lambda$ and
  $\mapsec$), while $\lambda\,H\,\partial_n f_\lambda = O(\sqrt\lambda)$
  and $\lambda\,\Delta_{\ambig_n} f_\lambda = O(\lambda)$ are lower order,
  yielding \eqref{eq:el_1d}.
  Outside the layer, $f_\lambda$ solves
  $f_\lambda - \lambda\,\Delta_{\mathcal{M}} f_\lambda = \mapsec$ with
  a smooth right-hand side and all terms $O(1)$; by standard elliptic
  regularity, $f_\lambda \to \mapsec$ pointwise as $\lambda\to 0$.
\end{proof}

\Cref{prop:bl} shows that the full manifold minimiser
is, to leading order, determined by the one-dimensional equation
$f - \lambda f'' = \mapsec$ in the normal coordinate $n$. \Cref{app:tikhonov_minimiser} solves this one-dimensional problem on $\bbR$ and shows
that $f_\lambda = \mapsec \ast K_\lambda$.

Working on $\bbR$ rather than
a bounded interval $(-\delta_0,\delta_0)$ is justified because the
Fermi chart extends to a fixed $\delta_0 > 0$ independent of $\lambda$,
while the transition layer has width $O(\sqrt\lambda) \ll \delta_0$.
The Green's function on $\bbR$ and on $(-\delta_0,\delta_0)$ with
Neumann conditions differ by an
exponentially small correction $O(e^{-\delta_0/\sqrt\lambda})$ in the transition layer.

\subsection{Minimiser of One-Dimensional Tikhonov Functional}
\label{app:tikhonov_minimiser}

Let $\mapsec\in L^2(\bbR;\bbR)$ and consider the one-dimensional
functional over $f \in H^1(\bbR;\bbR)$:
\[
  L_{1\mathrm{d}}(f) \;=\; \int_{\bbR}
  |f(x)-\mapsec(x)|^2 + \lambda|f'(x)|^2 \,\mathrm{d}x
\]
which is the effective functional of restricting to the normal $n$, as justified in
\cref{app:fermi_reduction}.

\begin{proposition}[1D convolution formula]
  \label{prop:1d_conv}
  The unique minimiser of $L_{1\mathrm{d}}$ is
  \begin{equation}\label{eq:conv_minimiser}
    f_\lambda \;=\; \mapsec \ast K_\lambda,
    \qquad
    K_\lambda(x) \;=\; \frac{1}{2\sqrt{\lambda}}\,e^{-|x|/\sqrt{\lambda}}.
  \end{equation}
\end{proposition}

\begin{proof}
  Taking the Fourier transform and applying Parseval's identity:
  \[
    L_{1\mathrm{d}}(f) \;=\; \int_{\bbR}
    |\hat{f}(\xi)-\hat{\mapsec}(\xi)|^2
    + \lambda\,\xi^2|\hat{f}(\xi)|^2 \,\mathrm{d}\xi.
  \]
  At each frequency $\xi$, the integrand is a strictly convex quadratic in
  $\hat f(\xi)$. The minimizer $\hat{f}_\lambda = \hat{\mapsec}\cdot(1+\lambda\xi^2)^{-1}$ becomes a convolution upon inverting:
  $f_\lambda = \mapsec \ast K_\lambda$, where
  \[
    K_\lambda(x)
    \;=\; \mathcal{F}^{-1}\!\!\left[\frac{1}{1+\lambda\xi^2}\right]\!(x)
    \;=\; \frac{1}{2\sqrt{\lambda}}\,e^{-|x|/\sqrt{\lambda}}.
  \]
\end{proof}

\begin{remark}
  Notice that the squared $L^2$ norm of $\mapsec$ is a constant to $f$ and $L_{1\mathrm{d}}(f)$ thus only depends on the linear action of $\mapsec$. Therefore, we can relax $g$ to be in the $H^{-1}(\bbR;\bbR)$ that will come in handy in the proof of \cref{prop:jac}.
\end{remark}

\subsection[Proof of Jacobian Blow-Up]{Proof of \cref{prop:jac} (Jacobian Blow-Up)}
\label{app:jac_proof}

\paragraph{Decomposition near a jump.}
Let $c_0 \in \ambig$ with jump $\Delta\mapsec = \mapsec^+ - \mapsec^-$.
Under the normal coordinate $n$ centred at $c_0$ (so that $c_0$
corresponds to $n=0$), we can write
$\mapsec(n) = \mapsec_{\mathrm{sm}}(n) + \Delta\mapsec\cdot H(n)$,
where $\mapsec_{\mathrm{sm}}$ is the $C^1$ part and $H$ is the Heaviside
function.
By linearity of convolution (\cref{prop:1d_conv}),
$
f_\lambda(n) \;=\; (\mapsec_{\mathrm{sm}} \ast K_\lambda)(n)
\;+\; \Delta\mapsec \cdot F_\lambda(n)
$
where
\[
  F_\lambda(s) \;=\; \int_{-\infty}^{s} K_\lambda(u)\,\mathrm{d}u
  \;=\;
  \begin{cases}
    \tfrac{1}{2}\,e^{s/\sqrt{\lambda}} & s \le 0, \\[2pt]
    1 - \tfrac{1}{2}\,e^{-s/\sqrt{\lambda}} & s \ge 0.
  \end{cases}
\]

\paragraph{First order derivative.}
Since $F_\lambda' = K_\lambda$, we have
\begin{equation}\label{eq:fderiv}
  f_\lambda'(n) \;=\; (\mapsec_{\mathrm{sm}}' \ast K_\lambda)(n)
  \;+\; \Delta\mapsec \cdot K_\lambda(n).
\end{equation}
Taking $n = 0$ yields
$  f_\lambda'(0) = (\mapsec_{\mathrm{sm}}' \ast K_\lambda)(0)
+ \Delta\mapsec \cdot K_\lambda(0)$.
The first term is $O(1)$ as $\lambda\to 0$ (smooth part), while the second
term $
\Delta\mapsec \cdot K_\lambda(0)
= \frac{\Delta\mapsec}{2\sqrt{\lambda}}
\longrightarrow\infty
$.
Hence $\|f_\lambda'(0)\| \sim |\Delta\mapsec|/(2\sqrt{\lambda})$.
By \cref{prop:bl}, the normal derivative of the full manifold
minimiser equals $f_\lambda'(0)$ to leading order,
proving \cref{prop:jac}.
Away from $\ambig$ (i.e.\ $|n| \gg \sqrt\lambda$),
$K_\lambda(n) \to 0$ exponentially, so $f_\lambda'(n)\to \mapsec'(n)$
pointwise. We showcase this blow-up behavior in \cref{fig:app_tikhonov}.

\paragraph{Error-jacobian near the jump.} The error satisfies
\begin{equation}\label{eq:fgerror}
  f_\lambda(n) - \mapsec(n) = \mapsec'' \mathcal{O}(\lambda) + \Delta \mapsec \cdot \sqrt{\lambda} K_\lambda(n)
\end{equation}
where the first term is of higher order compared to the second. Thus, comparing \cref{eq:fderiv} and \cref{eq:fgerror}, we have
$
| f_\lambda(n) - \mapsec(n) | = \sqrt{\lambda} |f_\lambda'(n)| + \mathcal{O}(\sqrt{\lambda})
$.

\begin{figure}[htbp]
  \centering
  \includegraphics[width=\linewidth]{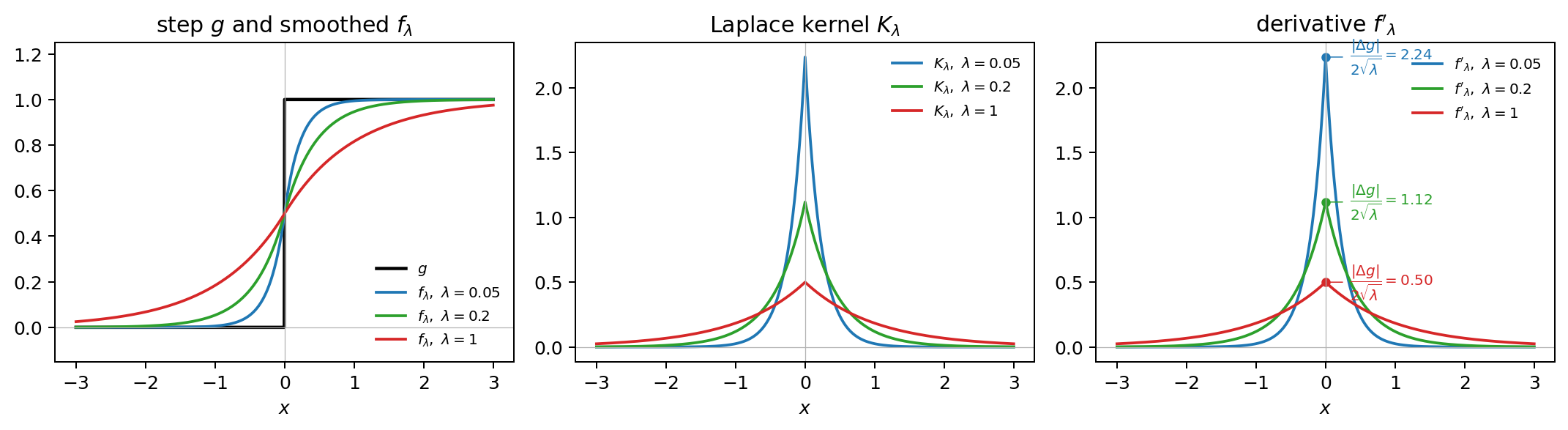}
  \caption{One-dimensional Tikhonov model.  As $\lambda \to 0$, $f_\lambda$
    sharpens toward $g$ (step function) and the Jacobian $\|Df_\lambda\|$ at the
  jump grows as $|\Delta g|/(2\sqrt{\lambda})$.}
  \label{fig:app_tikhonov}
\end{figure}




\subsection[Proof of Curvature Blow-Up]{Proof of \cref{prop:kink_curv} (Curvature Blow-Up)}
\label{app:curve}

We prove the local one-dimensional statement.  Write
$r:=\sqrt{\lambda}$, so that \[ K_\lambda(n)=\frac{1}{2r}e^{-|n|/r}.\]
Since the claim is local near the kink, we may either work on a bounded
normal interval with Neumann boundary conditions or multiply the model by a
smooth cutoff.
Let
\[\gamma_0(n)=
  \begin{cases}
    \gamma_-(n), & n<0,\\
    \gamma_+(n), & n>0,
  \end{cases}\qquad\gamma_-(0)=\gamma_+(0),
\]
with $\gamma_\pm\in C^2$ near $0$.  Denote
\[
  v_-:=\gamma_-'(0),\qquad
  v_+:=\gamma_+'(0),\qquad
  \Delta v:=v_+-v_-,
  \qquad
  v_m:=\frac{v_-+v_+}{2}.
\]
The one-dimensional Tikhonov smoothing is $\gamma_\lambda = K_\lambda * \gamma_0$.

We first compute the derivatives of $\gamma_\lambda$ at the kink.
Since $\gamma_0'$ has a jump discontinuity of size $\Delta v$ at $0$,
we may write, in a neighbourhood of $0$,
\[
  \gamma_0'(n)
  =
  v_-+\Delta v\,H(n)+\rho(n),
\]
where $H$ is the Heaviside function and $\rho(n)=O(|n|)$ as $n\to 0$.
Therefore
\[
  \gamma_\lambda'(n)
  =
  K_\lambda*\gamma_0'(n)
  =
  v_-+\Delta v\,F_\lambda(n)+O(r),
\]
where
\[
  F_\lambda(n)
  :=
  \int_{-\infty}^{n}K_\lambda(s)\,ds
  =
  \begin{cases}
    \frac12 e^{n/r}, & n\le 0,\\[2pt]
    1-\frac12 e^{-n/r}, & n\ge 0.
  \end{cases}
\]
In particular, $\gamma_\lambda'(0)= v_-+\frac12\Delta v+O(r)=v_m+O(r).$
Next, differentiating the previous expression gives the leading second
derivative:
\[
  \gamma_\lambda''(n)
  =
  \Delta v\,K_\lambda(n)+O(1).
\]
Equivalently, at the kink,
\[
  \gamma_\lambda''(0)
  =
  \frac{\Delta v}{2r}+O(1)
  =
  \frac{\Delta v}{2\sqrt{\lambda}}+O(1).
\]
The $O(1)$ term comes from the bounded classical second derivatives of
$\gamma_\pm$ and from the local cutoff or boundary correction. For a regular Euclidean curve, the Frenet curvature with arbitrary
parameter $n$ is
\[\kappa_{\gamma_\lambda}(n)=\frac{
  \|P_{\gamma_\lambda'(n)^\perp}\gamma_\lambda''(n)\|}
  {\|\gamma_\lambda'(n)\|^2},
\]
where
\[
  P_{w^\perp}z:= z-\frac{\langle z,w\rangle}{\|w\|^2}w
\]
is the orthogonal projection onto the complement of a nonzero vector $w$.
Assume $v_m\neq 0$.  Since
\[
  \gamma_\lambda'(0)=v_m+O(r), \qquad
  \gamma_\lambda''(0)=\frac{\Delta v}{2r}+O(1),
\]
the projection operator satisfies $P_{\gamma_\lambda'(0)^\perp} = P_{v_m^\perp}+O(r)$,
and hence
\[
  P_{\gamma_\lambda'(0)^\perp}\gamma_\lambda''(0)
  =
  \frac{P_{v_m^\perp}\Delta v}{2r}+O(1).
\]
Also $\|\gamma_\lambda'(0)\|^2 = \|v_m\|^2+O(r)$. Therefore
\[\kappa_{\gamma_\lambda}(0)=\frac{\|P_{v_m^\perp}\Delta v\|}
  {2r\,\|v_m\|^2}+O(1)
  =\frac{\|P_{v_m^\perp}\Delta v\|}
  {2\sqrt{\lambda}\,\|v_m\|^2}+O(1).
\]
If $P_{v_m^\perp}\Delta v\neq 0$, the leading term is nonzero, so the
curvature blows up at the smoothed kink with order $\lambda^{-1/2}$. More generally, for $|n|=O(r)$ one obtains the local profile
\[
  \kappa_{\gamma_\lambda}(n) =\frac{ K_\lambda(n)
  \bigl\|P_{v(n)^\perp}\Delta v\bigr\|}
  {\|v(n)\|^2} +O(1),
  \qquad
  v(n):=v_-+\Delta v\,F_\lambda(n).
\]
Thus the blow-up is localized to a normal layer of width
$O(\sqrt{\lambda})$ around the tangent discontinuity.

\section{Discontinuity in Prior-Mode: a Log-Gaussian Mixture Example}
\label{sec:gaussian-mixture}

We illustrate the emergence of the singularity in prior-mode $\mapsec$ via a multi-mixture gaussian example. The parameter space is $E:=\mathbb{R}_+^d$ with coordinates $\theta = (r_1, \dots, r_{d-1}, z)$ where $r_i > 0$ is the reflectance in channel $i$ and $z > 0$: illumination intensity. The rendering map $\pi: \mathbb{R}_+^d \to \mathbb{R}_+^{d-1}$ maps $(r_1, \dots, r_{d-1}, z)$ to $(r_1 z, \dots, r_{d-1} z)$.

\begin{proposition}
  \label{prop:2-mixture-gaussian}
  Assuming the prior distribution is a two-component mixture of log-Gaussians with equal weights
  $$p(r,z) := w_A \text{LogNormal}(\mu_A, \Sigma_A) + w_B \text{LogNormal}(\mu_B, \Sigma_B),$$
  where $w_A+w_B=1$ are positive weights, $\mu_A, \mu_B \in \mathbb{R}^d$ are the means in log-space and $\Sigma_{A,B} \in \mathbb{R}^{d \times d}$ are positive-definite covariance matrices. Then, the ambiguity locus $\Sigma$ is the level set of a up-to-second-order-form $\Phi$ via
  $\Phi(\log x) = \text{const}$
  and thus $\Sigma$ has co-dimension 1.
\end{proposition}

\begin{proof}
  Let $s_i = \log r_i$, $u = \log z$, $y_i = \log x_i$. The rendering map becomes additive
  $\tilde{\pi}(s_1, \dots, s_{d-1}, u) = (s_1 + u, \dots, s_{d-1} + u)$ and the prior in log-space is thus a Gaussian mixture
  $\tilde{p}(s, u) = w_A \mathcal{N}(\mu_A, \Sigma) + w_B \mathcal{N}(\mu_B, \Sigma)$.

  For any given log-transformed image $y = (y_1, \dots, y_{d-1}) \in \mathbb{R}^{d-1}$, the fiber is $\pi^{-1}(u) = \hat{y} + uv$, where $\hat{y} := (y, 0) \in \mathbb{R}^d$ and $v := (-1, \dots, -1, 1) \in \mathbb{R}^d$ is the fiber direction.
  The condition probability along the fiber reads
  $$\sum_{k=A,B} p_k(u \mid y) = \sum_{k=A,B}\frac{w_k}{(2\pi)^{d/2}|\Sigma_k|^{1/2}} \exp\!\left(-\frac{1}{2}(a_k + uv)^T M_k (a_k + uv)\right)$$
  with $a_k(y) := \hat{y} - \mu_k$ and $M_k := \Sigma_k^{-1}$.
  Each component is a 1D Gaussian with
  variance $\tilde{\sigma}_k^2 = \frac{1}{\alpha_k}$, center $u_k^*(y) = -\frac{\beta_k}{\alpha_k}$, and peak height $h_k^{(\log)} = w_k (2\pi)^{-d/2}|\Sigma_k|^{-1/2} \exp\!\left(-\frac{R_k^2(y)}{2}\right)$
  where $\alpha_k = v^T M_k v$, $\beta_k(y) = v^T M_k a_k$, and $R_k^2(y) = a_k^T M_k a_k - \beta_k^2/\alpha_k = a_k^T P_k^\perp a_k$ is the squared Mahalanobis distance from $\mu_k$ to the fiber, with $P_k^\perp = M_k - \frac{M_k vv^T M_k}{\alpha_k}$.

  In the $\mathbb{R}_+$ space, the Jacobian $1/z = e^{-u}$ shifts the mode to $u_k^{**} = u_k^* - \tilde{\sigma}_k^2$ and gives
  $$\log h_k^{(\mathbb{R}_+)} = \log w_k - \frac{1}{2}\log|\Sigma_k| - \frac{R_k^2(y)}{2} - u_k^*(y) + \frac{\tilde{\sigma}_k^2}{2} + C$$
  where $C$ is a universal constant.
  Then, the equal-peak condition (raising ambiguity) amounts to setting $\log h_A^{(\mathbb{R}_+)} = \log h_B^{(\mathbb{R}_+)}$, leading to
  $$\underbrace{\frac{1}{2}\bigl(R_A^2(y) - R_B^2(y)\bigr) + \bigl(u_A^*(y) - u_B^*(y)\bigr)}_{\Phi(y)} = \underbrace{\log\frac{w_A}{w_B} + \frac{1}{2}(\tilde{\sigma}_A^2 - \tilde{\sigma}_B^2) + \frac{1}{2}\log\frac{|\Sigma_B|}{|\Sigma_A|}}_{\phi_{A,B}}$$
  Here, $\phi_{A,B}$ depends on the mixture setup only and is independent of the probing location, so it is a constant in $y$. Thus, the geometry of the ambiguous locus solely depends on $\Phi(y)$, which is an up-to-second-order-form of $y$ since $\frac{R_k^2(y)}{2}$ is quadratic and $u_k^*(y)$ is linear in $y$. Then, under mild conditions that the coefficients are not degenerate, the level set of $\Phi(y)$ is either a conic hypersurface or a hyperplane, which has co-dimension 1.
\end{proof}

We illustrated the a specific case in \cref{fig:locus-calculated}. \Cref{prop:2-mixture-gaussian} can be readily applied to multi-log-gaussian-mixture case where the ambiguity is of co-dimension 1 locally, except on a null set where the conditional prior probability attains equal values for more than two mixtures.

\begin{figure}[htbp]
  \centering
  \includegraphics[width=0.8\textwidth]{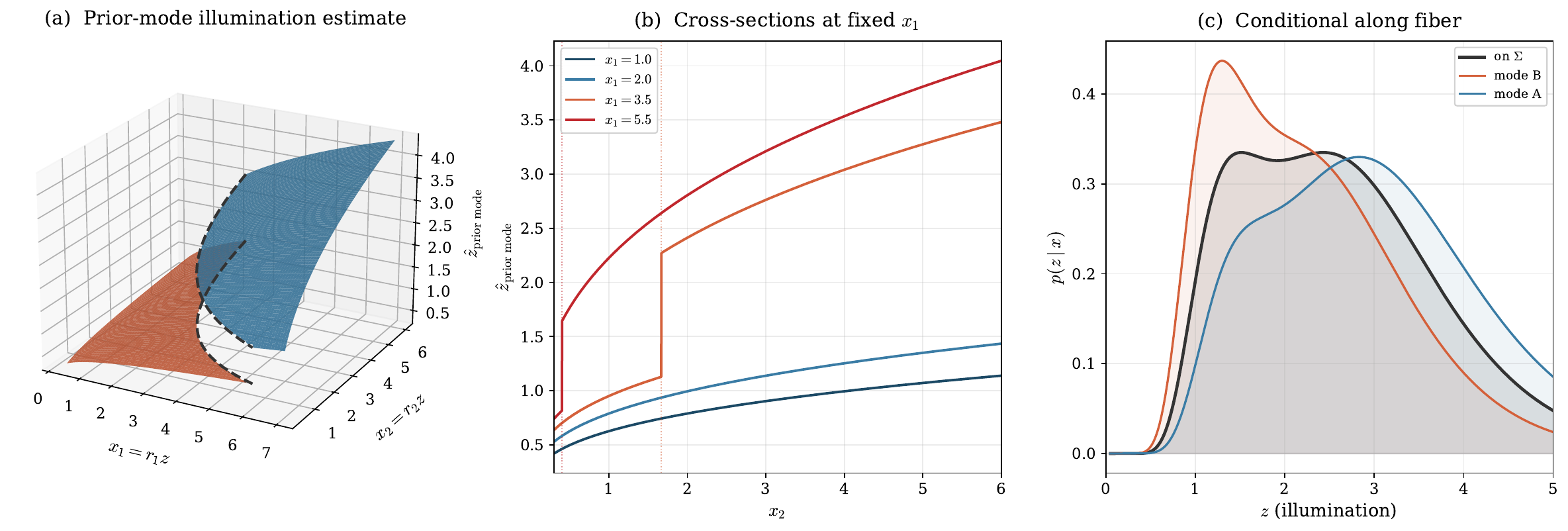}
  \caption{
    Ambiguous locus $\Sigma$ in $\mathbb{R}_+$ image space for a two-component log-Gaussian mixture prior. The two prior modes $\mu_A=(0.35, 0.40,  1.75)$ (blue) and $\mu_B=(-0.05, 0.75, -0.40)$ (orange) are in log-space with widths $\sigma_A=\sigma_B=0.55$. (a) Prior-mode illumination surface $\hat{i}_\mathrm{prior\ mode}(x)$, colored by dominant mode; the dashed curve is $\pi^{-1}(\Sigma)$, where both modes attain equal peak height. (b) Cross-sections of $\hat{i}_\mathrm{prior\ mode}$ at four fixed values of $x_1$; dotted verticals mark the $\Sigma$ crossing for each slice. (c) Conditional density $p(i\mid x)$ along the fiber at $x_1=e$, evaluated on $\Sigma$ (gray) and at offsets $\pm 0.3$ in $\log x_2$ toward mode B (orange) and mode A (blue).
  }
  \label{fig:locus-calculated}

\end{figure}
\section{Case Studies and Supplementary Images}
\subsection{CCT-sensitivity and Albedo Prediction Error}
\label{careaga}

\begin{figure}[p]
  \centering
  \includegraphics[width=\linewidth]{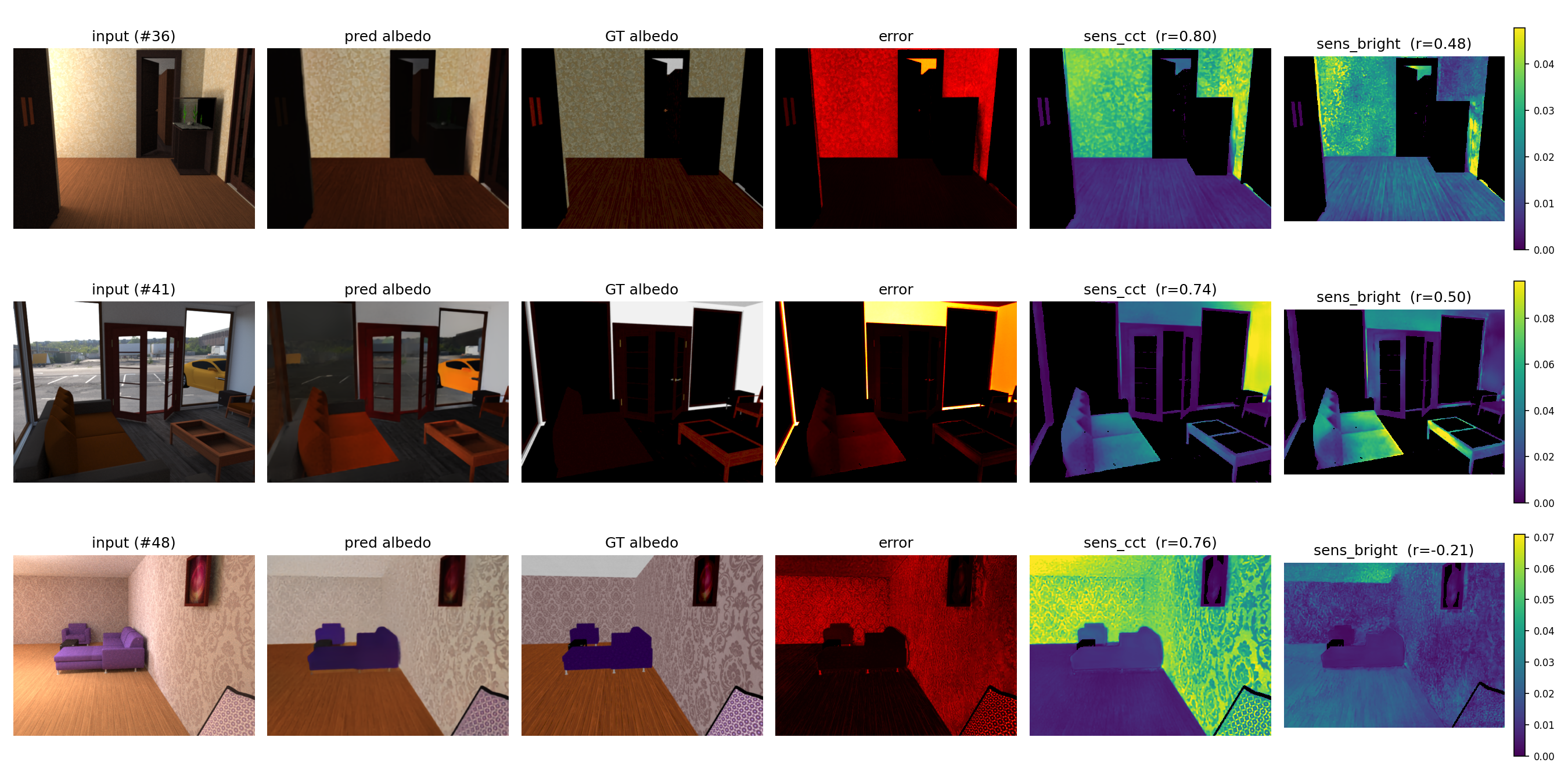}
  \caption{Per-pixel CCT-sensitivity heatmaps for high-$r_{\cct}$
  CGIntrinsics scenes. From left to right: input image, predicted albedo by Careaga DPT, ground truth albedo, per-pixel error norm, sensitivity under CCT perturbation $\|\partial \refl / \partial \cct\|$, and sensitivity under brightness perturbation. }
  \label{fig:heatmaps2}
\end{figure}

\begin{figure}[p]
  \centering
  \includegraphics[width=0.8\linewidth]{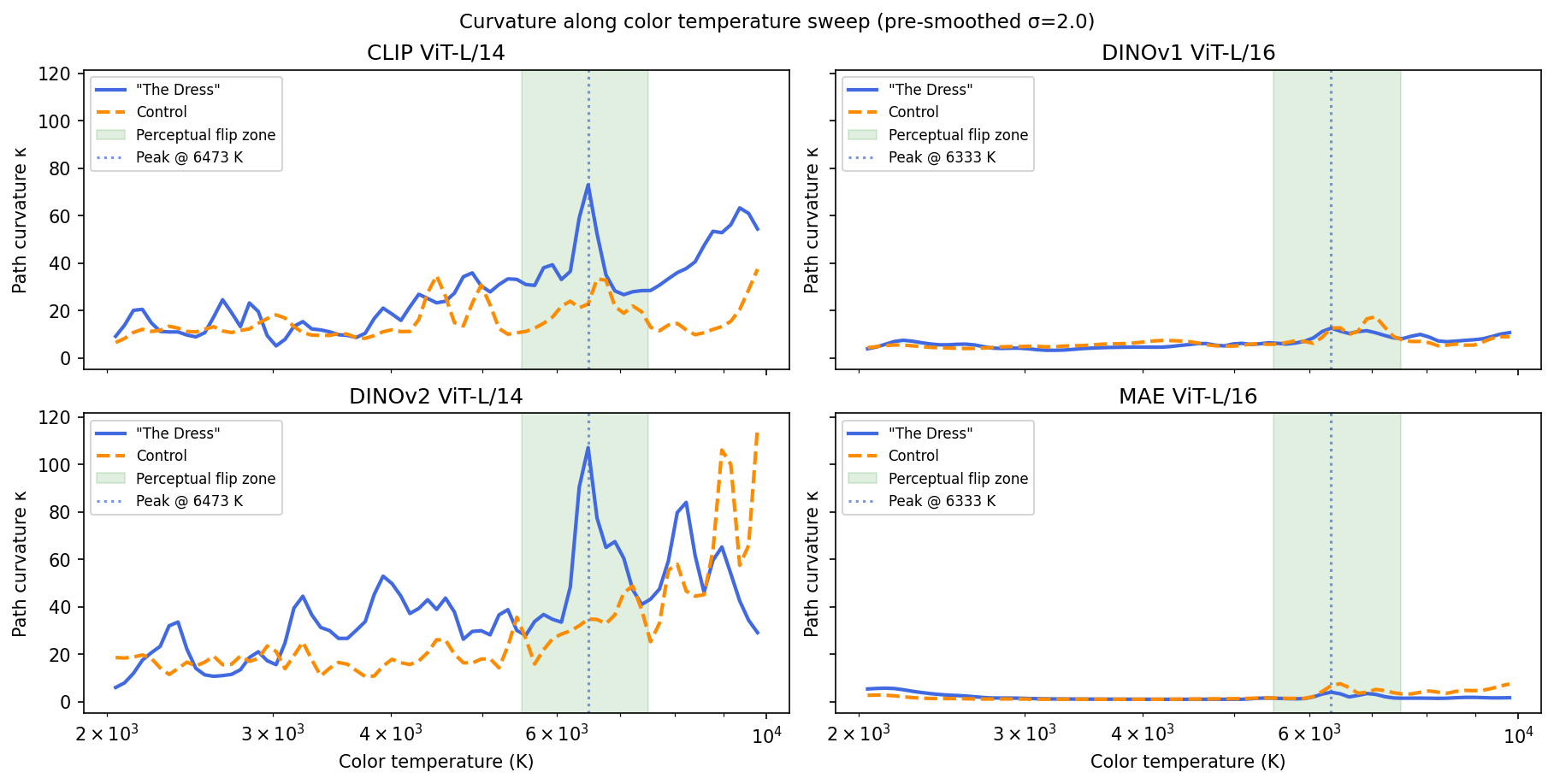}
  \caption{%
    The discrete Frenet-Serret curvature along latent path encoded by four vision models (labeled on top of the corresponding chart). The y-range is shared across all sub-figures; they are comparable since the length of the latent vector is normalized so the curvature is no longer sensitive to the scale of the embedding domain.
  }
  \label{fig:dress_curvature_full}
\end{figure}

Here, we showcase three testing cases and discuss why the CCT-perturbation is more pronounced than the brightness and saturation perturbation.
\Cref{fig:heatmaps2} shows representative per-pixel sensitivity heatmaps
for high-$r_{\cct}$ scenes: sensitivity concentrates on surfaces with
low chromaticity and low texture, precisely the regions where the
albedo-illumination decomposition is most underdetermined.

\subsection{Latent Curvature of The Dress and Control}
\label{app:case_studies}

We probe \textit{The Dress} and the control image in three more models:
DINOv1 ViT-L/16~\citep{caron2021emerging}, DINOv2 ViT-L/14~\citep{oquab2023dinov2},
and MAE ViT-L/16~\citep{he2022masked}. These span the dominant
self-supervised objectives for vision transformers: DINOv1 trains by
self-distillation, matching the softmax outputs of a student and an
exponential-moving-average teacher across augmented crops; DINOv2 scales
this with an added masked-token (iBOT) objective on curated data; and MAE
is a masked autoencoder trained to reconstruct pixels of occluded patches
under a reconstruction loss.

Each objective in Appendix~\ref{app:loss_taxonomy} is a proper loss whose population minimiser is the prior-mode
section $\mapsec$ (\cref{app:loss_taxonomy}), so each encoder approximates
$\mapsec$, and the implicit smoothing of training (weight decay,
augmentation, finite capacity) plays the role of the regulariser $\lambda$.
\Cref{prop:jac} then predicts the same consequence for every model: where
the colour-temperature sweep crosses the ambiguity locus $\ambig$, the
section jumps by $\Delta\mapsec$ and the smoothed encoder pays for it with a
latent-path second order derivative $f'' \sim \lvert\Delta\mapsec\rvert / (2\lambda)$
that spikes rather than the bounded curvature seen off $\ambig$.

This is
what \cref{fig:dress_curvature_full} shows: the semantically supervised
encoders (CLIP, DINOv2) exhibit a sharp \textit{Dress} curvature peak inside
the perceptual flip zone near D65, separated from the control; DINOv1 shows
the same qualitative bump with a weaker margin; and MAE, whose
pixel-reconstruction target is least aligned with the chromatic prior mode,
gives the flattest response of the four. The curvature blow-up therefore
tracks how strongly an objective commits to the prior-mode section, exactly
as the theory anticipates.

\end{document}